%% file: main.tex
\documentclass[preprint,12pt]{elsarticle}

\input{./tools/usepackages}

\input{./tools/macros}

\journal{Elsevier}

\begin{document}

\begin{frontmatter}

\title{K\H{o}nig = Ramsey \\ \Large A compactness lemma for Ramsey categories} 
\author[1]{Maximilian Hadek}
\affiliation[1]{organization={Department of Algebra, Charles University},
            addressline={Sokolovská 49/83}, 
            city={Prague 8},
            postcode={186 75}, 
            country={Czech Republic}}

\begin{abstract}
    We prove a new characterization of the Ramsey property of categories in terms of a generalized form of K\H{o}nig's tree lemma. Afterwards, we discuss its applications to structural Ramsey theory. In particular, we provide a new proof of the existence and uniqueness of minimal Ramsey expansions and a new transfer theorem which uses Grothendieck opfibrations and unifies several known Ramsey transfers.
\end{abstract}

\end{frontmatter}

\section{Introduction}

\input{content/intro}

\section{Preliminaries}
\label{sec:prelim}

\input{content/preliminaries}

\section{Proof of Theorem~\ref{thm:ramsey=könig}}
\label{sec:proof}

\input{content/mainProof}

\section{Connections to Structural Ramsey Theory}
\label{sec:structures}

\input{content/structures}

\section{Transferring the Ramsey property}
\label{sec:transfers}

\input{content/transfers}

\section*{Acknowledgments}

\input{content/acknowledgements}

\bibliographystyle{elsarticle-num}

\end{document}

%% file: tools/usepackages.tex
\usepackage{microtype}

\usepackage[mathcal]{euscript}

\usepackage{amsmath}
\usepackage{amsfonts}
\usepackage{amssymb}
\usepackage{amsthm}
\usepackage{tikz-cd}

\usepackage{thmtools}

\declaretheorem[
	name=Theorem,
	numberwithin=section
	]{thm}
\declaretheorem[
	name=Lemma,
	sibling=thm,
	]{lem}
\declaretheorem[
	name=Proposition,
	sibling=thm,
	]{prop}
\declaretheorem[
	name=Corollary,
	sibling=thm,
	]{cor}
\declaretheorem[
	name=Definition,
	style=definition,
        sibling=thm
	]{defin}
\declaretheorem[
	name=Remark,
	style=remark,
	numbered=no
	]{rem}	
\declaretheorem[
	name=Example,
	style=remark,
        sibling=thm
	]{exam}

\usepackage[pdftex,pdfpagelabels]{hyperref}
\hypersetup{pdftitle={K\H{o}nig=Ramsey},
		pdfauthor={Max Hadek}}

%% file: tools/macros.tex
\newcommand{\cat}[1]{\mathcal{#1}}
\newcommand{\str}[1]{\mathbb{#1}}

\newcommand{\op}{^{\mathop{op}}}

\DeclareMathOperator{\Set}{Set} 
\DeclareMathOperator{\Cat}{Cat}

\DeclareMathOperator{\Elts}{\int\!}
\DeclareMathOperator{\id}{id}
\DeclareMathOperator{\im}{im}
\DeclareMathOperator{\ar}{ar\!}

%% file: content/intro.tex
Given some sets and some maps between them, is there a selection of elements, one from each set, which is compatible with all given maps?
The simplest case, when there are no maps, is precisely the axiom of choice: given any family of nonempty sets, one can pick a single element out of each of them. Slightly less simple, K\H{o}nigs infinity lemma deals with the case when the sets and maps are arranged in sequence.

\begin{thm}[K\H{o}nig, 1927 \cite{könig1927}]
    Consider finite nonempty sets $E_1,E_2,E_3,\dots$ and maps $E_{n+1}\xrightarrow{R_n}E_n$. We can find in every set $E_n$ an element $a_n$ such that $R_n(a_{n+1}) = a_n$ for all $n\in\mathbb{N}$.\footnote{While this is almost the original wording of the lemma, some readers might be more familiar with a graph theoretic formulation in terms of trees. To see the similarity, draw all elements of all $E_n$ as vertices and connect them by an edge whenever $R_n(a_{n+1})=a_n$.}
    
\end{thm}
More complicated shapes in which sets and maps can be arranged are categories $\cat C$, which are collections of points and arrows together with a formal composition of arrows. A diagram in the shape of $\cat C$ assigns to every point a set and to every arrow a map, such that composition is respected, and \emph{solutions} to a diagram are the selections of elements, one from each set, which are invariant under the given maps.
Our main theorem characterizes those categories, for which a generalized K\H{o}nigs lemma holds, i.e. those shapes where every diagram of finite nonempty sets has a solution.

\begin{thm}
\label{thm:ramsey=könig}
    If $\cat C$ is an essentially small and locally finite category, then the following are equivalent: 
\begin{enumerate}
    \item every diagram in the shape of $\cat C\op$ which consists of finite nonempty sets has a solution
    \item $\cat C$ is confluent and has the Ramsey property.
\end{enumerate}
\end{thm}

The symbol $\op$ means that the maps of the diagram point in the \emph{opposite} direction as the arrows in $\cat C$.

After proving Theorem~\ref{thm:ramsey=könig} in Section~\ref{sec:proof}, we discuss its applications to structural Ramsey theory. In particular, in Section~\ref{sec:structures} we generalize the notion of expansions to arbitrary categories and prove the following refinement of a theorem which initially appeared in the context of topological dynamics; see \cite{KPT2005} and \cite{NVT2013precompactExpansions}. 
\begin{thm}
\label{thm:introMinimal}
    If a locally finite and essentially small confluent category has a precompact confluent Ramsey expansion, then it has a precompact confluent Ramsey expansion with the expansion property. Moreover, this expansion is unique up to isomorphism of expansions.
\end{thm}
Moreover, in Section~\ref{sec:transfers} we prove a new Ramsey transfer theorem about Grothendieck opfibrations, providing a common generalization of three known Ramsey transfers, namely products \cite{sokic2012ramsey}, discrete opfibrations \cite{mavsulovic2023fibrations} and blowups.
\begin{thm}
\label{thm:introOpfibration}
    If a Grothendieck opfibration is locally finite and essentially small, has confluent Ramsey fibers and a confluent Ramsey base, then it is confluent and Ramsey.
\end{thm}

%% file: content/preliminaries.tex
\subsection{Categories, Diagrams and Limits}

Recall the elementary category theoretical notions required for Theorem~\ref{thm:ramsey=könig}. See also any book on the subject, for example \cite{borceux1994handbook} or \cite{Riehl2016context}. 

\begin{defin}
    A \emph{category} $\cat C$ is a collection of objects and arrows between them, together with a formal composition operation of arrows. Objects are denoted as $C\in\cat C$ and for an arrow $f$ we specify its source and target object by writing $A\xrightarrow{f}B$. Given three objects $A,B,C$ and two arrows $A\xrightarrow{f}B$ and $B\xrightarrow{g}C$ there is a composite arrow $A\xrightarrow{g\circ f}C$, so that the following axioms are fulfilled.
    \begin{itemize}
        \item (Identity) For every object $B\in \cat C$ there is an arrow $B\xrightarrow{\id_B}B$ which fulfills $f=\id_B\circ f$ and $g= g\circ \id_B$ for all arrows $A\xrightarrow{f}B$ and $B\xrightarrow{g}C$.
        \item (Associativity) For any three arrows, $A\xrightarrow{f}B$, $B\xrightarrow{g}C$ and $C\xrightarrow{h}D$ we have $h\circ(g\circ f) = (h\circ g)\circ f$.
    \end{itemize}
\end{defin}
Throughout the article, all considered categories (except $\Set$ and $\Cat$) are assumed to be small, which means that their objects and morphisms form a set and not a proper class. However, since the discussed properties are preserved by equivalence of categories, all theorems hold for \emph{essentially small} categories, i.e.\! those which are equivalent to a small category. For example, the classes of relational structures considered in Section~\ref{sec:structures} are typically not small, but essentially small. 

\begin{defin}
    A set valued \emph{diagram} $\cat D:\cat C\op\to\Set$ is a functor, it assigns to every object $C\in\cat C$ a set $\cat D_C$ and to each arrow $A\xrightarrow{f}B$ in $\cat C$ a map $\cat D_B\xrightarrow{\cat D_f}\cat D_A$ in the opposite direction. Moreover, this assignment must respect the identity arrows and the composition of arrows, meaning that
    $$
    \cat D_{\id_C} = \id_{\cat D_C} \quad\text{and}\quad 
    \cat D_f\circ \cat D_g = \cat D_{g\circ f}
    $$
    for any three objects $A,B,C$ and arrows $A\xrightarrow{f}B$ and $B\xrightarrow{g}C$.
\end{defin}

\begin{defin}
    Let $\cat C$ be a category, and let $\cat D:\cat C\op\to\Set$ be a diagram.
    A \emph{solution} to $\cat D$ is a tuple $(x_C)_{C\in\cat C}$, where each $x_C$ is an element of the set $\cat D_C$ and $x_A=\cat D_f(x_B)$ holds for all arrows $A\xrightarrow{f}B$ in $\cat C$.
    The set of all solutions is called \emph{limit} and is denoted as $\lim\cat D$.
    $$
    \lim \cat D = 
    \big\{ (x_C)_{C\in\cat C} \bigm|\forall (A\xrightarrow{f}B) :x_A = \cat D_f(x_B) \big\}
    \subseteq\prod_{C\in\cat C}\cat D_C
    $$
\end{defin}

\subsection{Connectedness and Confluence}

\begin{defin}
\label{def:connected}
    We say that two objects $A$ and $B$ in a category $\cat C$ are \emph{connected}, if there is a zigzag path of arrows between them.
$$
\begin{tikzcd}[column sep={3em,between origins},row sep={2em,between origins}]
  & C_1 \arrow[rd] \arrow[ld] &         & \cdots \arrow[ld] \arrow[rd] &         & C_n \arrow[ld] \arrow[rd] &   \\
A &                               & C_2 &                              & C_{n-1}&                               & B
\end{tikzcd}
$$
A category is said to be connected if any two of its objects are connected.
\end{defin}

Since there are no arrows between different connected components of a category, to find solutions to a diagram $\cat D:\cat C\op\to\Set$, it suffices to consider each component individually. Indeed, if $\cat C =\coprod_{i}\cat C_i$ is a decomposition of $\cat C$ into its connected components, then
$$
\lim\cat D = \prod\lim_i\cat D|_{\cat C_i}
$$
where $\cat D|_{\cat C_i}$ is the restriction of $\cat D$ to the $i$-th component of $\cat C$.

\begin{exam}
\label{exam:confluence}
Not every diagram has a solution.
$$
\begin{tikzcd}
\{0\} \arrow[r, hook] & {\{0,1\}} & \{1\} \arrow[l, hook]
\end{tikzcd}
$$
Let $A\leftarrow D\to B$ be the shape of this diagram, then sets $\cat D_A$ and $\cat D_B$ are independent, as are the maps leaving them. This can be circumvented by requiring the existence of an object $C$ and two arrows $A\to C\leftarrow B$.
\end{exam}

\begin{defin}
\label{def:confluent}

A category $\cat C$ is called \emph{confluent} if for any two connected objects $A$ and $B$, there is a third object $ C$ and two arrows $A\to C\leftarrow B$.
\end{defin}

    A common assumption in Ramsey theory is that categories have the so-called \emph{joint embedding property}, which says that for any two objects$A$ and $B$ there must be an object $C$ and two arrows $A\to C\leftarrow B$. Hence, a category is confluent if all of its connected components have the joint embedding property. The following proposition generalizes Example~\ref{exam:confluence}, showing that confluence is indeed necessary for Theorem~\ref{thm:ramsey=könig}.

\begin{prop}[K\H{o}nig $\Rightarrow$ confluence]
\label{prop:confluence}
    If $\cat C$ a category is not confluent, then there is a diagram of finite nonempty sets in the shape of $\cat C\op$, which does not have a solution.
\end{prop}

\begin{proof}
    If $\cat C$ is not confluent, then we can find two connected objects $A$ and $B$ such that there are no objects $C$ with arrows $A\to C\leftarrow B$. 
    We define a diagram $\cat D:\cat C\op\to\Set$ as follows.
    $$
    C\mapsto 
    \begin{cases}
    \{0\} & \text{if } A\xrightarrow{\exists} C \\
    \{1\} & \text{if } B\xrightarrow{\exists} C \\
    \{0,1\} & \text{otherwise}
    \end{cases}
    $$
    The assumption guarantees that the first two cases are disjoint.
    Given an arrow $C\xrightarrow{f}C'$, we have $\cat D_{C'}\subseteq\cat D_C$, hence we can define $\cat D_f$ to always be the inclusion map. Restricting $\cat D$ to a path from $A$ to $B$ yields the following diagram.
$$
    \begin{tikzcd}[column sep={4em,between origins},row sep={3em,between origins}]
  & \cat D_{C_1} &         & \cdots  &         & \cat D_{C_n}&   \\
 \{0\} =\cat D_A\arrow[ru, hook] &   & \cat D_{C_2} \arrow[lu, hook]\arrow[ru, hook]&    & \cat D_{C_{n-1}} \arrow[lu, hook]\arrow[ru, hook]&                               & \cat D_B = \{1\}\arrow[lu, hook]
\end{tikzcd}
$$
All maps are inclusions, so any solution would have to pick $0$ out of $\cat D_A$, $\cat D_{C_1}$, $\cat D_{C_2}$ and so on, eventually contradicting $\cat D_B =\{1\}$.
\end{proof}

\subsection{The Ramsey Property}
\begin{exam}
\label{exam:Ramsey}
The following diagram does not have a solution.
$$
\begin{tikzcd}
\{*\} \arrow[r, "0", shift left] \arrow[r, "1"', shift right] & {\{0,1\}}
\end{tikzcd}
$$
\end{exam}
 
\begin{defin}
\label{def:ramsey}
    A category\footnote{Often, when defining the Ramsey property, one assumes that the category is \emph{locally finite}, which means that there are only finitely many arrows between any two objects. This assumption is also necessary for Theorem~\ref{thm:ramsey=könig}, however we are curious whether there is an appropriate adaptation of the Ramsey property for non locally finite categories, such that Theorem~\ref{thm:ramsey=könig} holds.} 
    $\cat C$ is said to be \emph{Ramsey} if for any two objects $A$ and $B$ and any finite set $N$ there exists an object $C$, such that for all maps $\hom(A,C)\xrightarrow{\chi} N$ there exists an arrow $B\xrightarrow{g} C$ such that the composition
    $$
   \chi\circ g_*:\hom(A,B)\xrightarrow{g_*} \hom(A,C)\xrightarrow{\chi} N
    $$
    is a constant map, where $g_*$ is the composition map $f\mapsto g\circ f$. Any such object $C$ is called a \emph{Ramsey witness} of $A,B$ and $N$.
\end{defin}
 
\begin{exam}
\label{exam:ramseysey}
    Any diagram $\cat D:\cat C\op\to\Set$ induces a map $\chi$ in the following way. Let $N:=\cat D_A$, pick an element $x\in\cat D_C$ and define $\chi$ as the following map.
    $$
    \chi:\hom(A,C)\xrightarrow{}\cat D_A,\quad (A\xrightarrow{h}C)\mapsto\cat D_h(x)
    $$
    Let $B\xrightarrow{g}C$ be the arrow as in the definition above, then $\cat D_g(x)\in\cat D_B$ is sent to the same element in $\cat D_A$ by all maps $\cat D_f$ coming from arrows $A\xrightarrow{f}B$.
$$
\begin{tikzcd}
\cat D_C \arrow[r, "\cat D_g"] & \cat D_B \arrow[r, "\cat D_{f'}"', shift right] \arrow[r, "\cat D_f", shift left] & \cat D_A
\end{tikzcd}
$$
Indeed $\cat D_f(\cat D_g(x)) = \cat D_{g\circ f}(x) = \chi(g\circ f)$ which does not depend on $f$. In particular, there is a selection of elements from each $\cat D_C$, which respects all arrows from $A$ to $B$.
\end{exam}

%% file: content/mainProof.tex
\subsection{Ramsey implies K\H{o}nig}

Throughout the proof, we use the following notation. Given a set $M$ of arrows in $\cat C$, we call a tuple $(x_C)_{C\in\cat C}$ an $M$\emph{-solution} of a diagram $\cat D$, if the tuple is compatible with all arrows in $M$. The set of $M$-solutions is denoted as $\lim_M\cat D$.
$$
\lim\nolimits_M \cat D :=\big\{ (x_C)_{C\in\cat C} \bigm|\forall (A\xrightarrow{f}B) \in M: x_A = \cat D_f(x_B)  \big\} \subseteq \prod_{C\in\cat C}\cat D_C
$$
We start by iterating the technique from Example~\ref{exam:ramseysey} to obtain $M$-solutions whenever $M$ is a finite set of arrows.

\begin{prop}
\label{prop:iterateRamsey}
    Let $\cat C$ be a connected confluent Ramsey category, let $M$ be a finite set of arrows in $\cat C$, and let $\cat D:\cat C\op\to \Set$ be a diagram. If all sets $\cat D_C $ are finite and nonempty, then $\cat D$ has an $M$-solution.
\end{prop}

\begin{proof}
Let $B_1,\dots,B_n$ be a list of all objects which are the source or target of some arrow in $M$. Since $\cat C$ is connected and confluent, there exists an object $C_0\in \cat C$ together with arrows $B_i\xrightarrow{g_i} C_0$ for every index $i\in\{1,\dots,n\}$.

Let $C_1$  be a Ramsey witness of the pair $B_1$, $C_0$ and the finite set $\cat D_{B_1}$. 
    Let $C_2$ be a Ramsey witness of the pair $B_2$, $C_1$ and the set $\cat D_{B_2}$ and iterate until we defined $C_n$ as a Ramsey witness of $B_n, C_{n-1}$ and $\cat D_{B_n}$. Then, pick any element $x_n\in \cat D_{C_n}$ and consider the following map. 
    $$
    \chi_n : \hom(B_n,C_n)\to \cat D_{B_n},\quad f\mapsto\cat D_f(x_n)
    $$
    As $C_n$ is a Ramsey witness of $B_n$, $ C_{n-1}$ and the set $\cat D_{B_n}$, there exists an arrow $C_{n-1}\xrightarrow{h_n} C_n$ such that the map
    $$
    \hom(B_n, C_{n-1})\to \cat D_{B_n},\quad f\mapsto 
    \chi_n(h_n \circ f) =
    \cat D_f ( \cat D_{h_n}(x_n) ) 
    $$
    is constant.
    Set $x_{n-1}:= \cat D_{h_n}(x_n)$ and iterate backwards, until we arrive at the following diagram.
    $$
    \begin{tikzcd}
    &B_1 \arrow[rd, "g_1"]              &                      &                      &    &                    &     \\
    \arrow[rr, "\vdots", phantom, shift left=1]& \arrow[r, "\vdots", phantom, shift left=1] & C_0 \arrow[r, "h_1"] & C_1 \arrow[r, "h_2"] & \cdots \arrow[r, "h_n"] & C_n &\\
    &B_n \arrow[ru, "g_n"']             &                      &                      &                        &    &
    \end{tikzcd}
    $$
    We claim that any tuple $(x_C)_{C\in\cat C}$ with $x_{B_i}=\cat D_{h_n\circ\dots\circ h_1\circ g_i}(x_n)$ is an $M$-solution. 
    Indeed, given any arrow $B_i\xrightarrow{f} B_j$ from $M$, the compositions 
    $$
    f_0:= h_{i-1}\circ\dots \circ h_1\circ g_j \circ f \quad\text{and}\quad f_1:= h_{i-1}\circ\dots\circ h_1\circ g_i
    $$
    are arrows from $B_i$ to $C_{i-1}$ and therefore, by the definition of $h_i$, we have
    \begin{align*}
    \cat D_{f_0}(\cat D_{h_i}(x_i))&= 
    \chi_i(h_i\circ f_0) \\&= 
    \chi_i(h_i\circ f_1)=
    \cat D_{f_1}(\cat D_{h_i}(x_i)).
    \end{align*}
    Plugging in the definitions of $f_0$, $f_1$ and $x_i$ we get
$$
        \cat D_f\big( \cat D_{h_n\circ\dots\circ h_1\circ g_j}( x_n ) \big) 
        =
        \cat D_{h_n\circ\dots\circ h_1\circ g_i}( x_n ) 
$$
    which proves the claim.
\end{proof}

\begin{lem}[Compactness]
\label{lem:compactness1}
 Let $\cat C$ be a small category and let $\cat D : \cat C \op\to \Set$ be a diagram where all sets $\cat D_C$ are finite. If $\cat D$ has no solution, then there is a finite set $M$ of arrows in $\cat C$ such that $\cat D$ does not even have an $M$-solution.
\end{lem}

\begin{proof}
     Consider the sets $\cat D_C$ as discrete (and, as they are finite, compact) topological spaces and take their product space
    $
    \prod_{C\in \cat C} \cat D_C
    $.
    By Tychonoff's theorem, this space is compact and therefore every family of closed subspaces with empty intersection has a finite subfamily with empty intersection. 
    The limit of $\cat D$ can be written as the following intersection of closed subspaces.
    $$
    \lim \cat D = 
    \bigcap_{A\xrightarrow{f} B}
    \{(x_C)_{C\in \cat C} \mid x_A = \cat D_f(x_{B}) \}
    $$
    If this intersection is empty, then there is a finite set $M$ of arrows, such that 
    $$
    \lim\nolimits_M\cat D = \bigcap_{f\in M} \{(x_C)_{C\in \cat C} \mid x_A = \cat D_f(x_{B}) \}
    $$
    is also empty.
\end{proof}

\begin{cor}[Ramsey $\Rightarrow$ K\H{o}nig]
\label{cor:compactRamsey}
    Let $\cat C$ be a small confluent Ramsey category and let $\cat D:\cat C\op\to \Set$ be a diagram. If all sets $\cat D_C$ are finite and nonempty, then $\cat D$ has a solution.
\end{cor}

\begin{proof}
    Let $\cat C = \coprod_i \cat C_i$ be the decomposition of $\cat C$ into its connected components, and let $\cat D|_{\cat C_i}$ be the respective components of the diagram $\cat D$. As there are no arrows between different components, we can compute the limit of $\cat D$ component-wise.
    $$
    \lim\cat D = \prod_i \lim\cat D|_{\cat C_i}
    $$
    Every component $\cat C_i$ is connected, confluent and Ramsey, so by Proposition~\ref{prop:iterateRamsey} there is an $M$-solution of $\cat D|_{\cat C_i}$ for any finite set $M$ of arrows in $\cat C_i$.  Lemma~\ref{lem:compactness1} then implies that the limit of $\cat D|_{\cat C_i}$ is nonempty.
\end{proof}

\subsection{K\H{o}nig implies Ramsey}
The following proposition requires that the considered category $\cat C$ be \emph{locally finite}, which means that between any two objects there are only finitely many arrows.

\begin{prop}[K\H{o}nig $\Rightarrow$ Ramsey]
    Let $\cat C$ be a locally finite category such that every diagram $\cat D:\cat C\op\to\Set$ where all sets $\cat D_C$ are finite and nonempty has is a solution. Then $\cat C$ has the Ramsey property.
\end{prop}

\begin{proof}
Fix two objects $A$ and $B$, a finite set $N$ and consider the diagram $\cat D$, which assigns to every object $C$ the set of maps $\chi$, which prevent $C$ form being a Ramsey witness of $A$, $B$ and $N$.
\begin{align*}
    \cat D_C &:= \{\hom(A,C)\xrightarrow{\chi} N\mid \forall (B\xrightarrow{h} C): \chi\circ h_* \text{ is not constant}\} \\
    \cat D_f &: \cat D_{C}\to \cat D_{C'}, \quad \chi\mapsto \chi\circ f_*
\end{align*}
for all objects $C, C'$ and arrows $C\xrightarrow{f}C$. To check that $\cat D_f$ is well defined, take a map $\chi\in\cat D_{C_2}$, an arrow $B\xrightarrow{h}C_1$ and observe that the map
$$
\cat D_f(\chi) \circ h_* = \chi\circ f_*\circ h_* = \chi\circ(f\circ h)_* 
$$
is not constant, hence $\cat D_f(\chi)$ is indeed an element of $\cat D_{C_1}$.
Observe that every set $\cat D_C$ is finite because $\hom(A,C)$ is finite and that $\cat D_C$ is empty if and only if $C$ is a Ramsey witness of $A$, $B$ and $N$. Hence, the Ramsey property will follow once we show that $\cat D$ does not have a solution.

Striving for a contradiction, assume that $(\chi_C)_{C\in\cat C}$ is a solution, which in particular means that $\cat D_f(\chi_B) = \chi_{A}$ for all arrows $A\xrightarrow{f}B$. The following computation shows that $\chi_B$ must be a constant map.
$$
\chi_B(f) = 
(\chi_B\circ f_*)(\id_{A}) = 
\cat D_f(\chi_B)(\id_{A}) =
\chi_{A}(\id_{A})
$$
Therefore $\chi_B\circ (\id_B)_*$ is constant and $\chi_B$ cannot be an element of $\cat D_B$, giving the desired contradiction. 
\end{proof}

%% file: content/structures.tex
\subsection{Structures and Expansions}
We connect Theorem~\ref{thm:ramsey=könig} to established Ramsey theory using a notion of \emph{abstract expansions} of a Ramsey class. First, recall some model theoretic notions commonly used in Ramsey theory, see for example a recent survey by Hubi\v cka and Kone\v cn\'y \cite{hubivcka2025twenty}.

\begin{defin}
    A \emph{relational language} $\sigma$ is a set of symbols $R$, each of which has an associated finite set $\ar R$, called its arity. A finite $\sigma$-\emph{structure} $\str C$ is a tuple $(C; R^\str C \mid R\in\sigma)$ where $C$ is a finite set and each $R^{\str C}$ is a subset of $C^{\ar R}$. 
    
    Given two $\sigma$-structures $\str C$ and $\str C'$ we say that an injective map $C\xrightarrow{f} C'$ is an \emph{embedding} from $\str C$ to $\str C'$ if for any symbol $R\in\sigma$ we have $\Bar{c}\in R^{\str C}$ if and only if $f\circ\Bar{c}\in R^{\str C'}$ for all tuples $\ar R\xrightarrow{\Bar{c}}C$. 
    
    When $\cat C$ is a class of finite $\sigma$-structures, we consider it as a category using the embeddings as arrows. Any such category is locally finite and essentially small; hence, Theorem~\ref{thm:ramsey=könig} applies.
\end{defin}

\begin{defin}
\label{def:expansionExplicit}
    Let $\sigma$ and $\tau$ be two disjoint relational languages, let $\cat C$ be a class of $\sigma$-structures and let $\cat E$ be a class of $\sigma\cup\tau$-structures. Given a structure $\str E\in\cat E$ we denote by $\str E|_\sigma$ the $\sigma$-structure, which is obtained from $\str E$ by forgetting all relations from $\tau$. We say that $\cat E$ an \emph{expansion} of $\cat C$ if 
        $$
        \cat C = \{ \str E|_\sigma\mid \str E\in\cat E\}.
        $$ 
    An expansion called \emph{precompact} if for every structure $\str C\in\cat C$ there are only finitely many structures $\str E\in\cat E$ with $\str E|_{\sigma}= \str C$.
\end{defin}

To relate expansions to the diagrams in Theorem~\ref{thm:ramsey=könig}, we generalize the notion of expansion from classes of relational structures to arbitrary categories.

\begin{defin}
\label{def:expansions}
    Let $\cat C$ be a category. A functor $\cat E\xrightarrow{\pi}\cat C$ is called \emph{(abstract) expansion} of $\cat C$ if
    \begin{enumerate}
        \item (surjectivity) for every $C\in\cat C $ there is $E\in\cat E$ with $C=\pi(E)$ and 
        \item (discrete fibration) for every $E\in \cat E$ and every arrow $C'\xrightarrow{f} \pi(E)$ in $\cat C$, there is a unique arrow $E'\xrightarrow{g}E$ in $\cat E$, such that $C'=\pi(E')$ and $f=\pi(g)$.
    \end{enumerate}
    Again, we say that $\pi$ is \emph{precompact}\footnote{Another term for 'precompact' would be 'locally finite'.} if for every object $C\in\cat C$ there are only finitely many $E\in\cat E$ with $\pi(E)=C$.
\end{defin}

Functors which satisfy the second axiom of expansions are often referred to as \emph{discrete Grothendieck fibrations} or simply \emph{discrete fibrations}, see, for example, \cite{LoregianRiehl2020fibrations} or \cite[Section~2.4]{Riehl2016context}.

\begin{lem}[Discrete Grothendieck construction]
\label{lem:discreteGrothendieck}
    There is a one-to-one correspondence between abstract expansions of $\cat C$ and diagrams in the shape of $\cat C$ consisting of nonempty sets. Moreover, expansions are precompact if and only if the sets in the corresponding diagram are finite.
\end{lem}

\begin{proof}[Proof sketch]
    Given an expansion $\cat E\xrightarrow{\pi}\cat C$, the corresponding diagram sends objects $C\in\cat C$ to the set $\pi^{-1}(C)$. To arrows $C'\xrightarrow{f}C$ we assign a map $\pi^{-1}(C)\to \pi^{-1}(C')$, which sends objects $E\in\pi^{-1}(C)$ to the unique $E'$ given by the second axiom of expansions.

    Conversely, if $\cat D:\cat C\op\to\Set$ is a diagram, let the objects of the corresponding $\cat E$ be all elements of all sets $\cat D_C$.
    Arrows between two elements $x'\in\cat D_{C'}$ and $x\in\cat D_C$ are triples $(x',x, f)$ where $f$ is an arrow from $C'$ to $C$ and $\cat D_f(x)=x'$.
    The functor $\cat E\xrightarrow{\pi}\cat C$ sends $x$ to $C$ whenever $x\in\cat D_C$ and arrows $(x',x,f)$ to $f$. 
\end{proof}

\begin{cor}
\label{cor:structures}
    A class of finite relational structures $\cat C$ is confluent and Ramsey if and only if for every precompact expansion $\cat E\xrightarrow{\pi}\cat C$ there is a functor $\cat C\xrightarrow{\alpha} \cat E$ such that $\pi\circ\alpha = \id_{\cat C}$.
    $$
    \begin{tikzcd}
    \cat C \arrow[rd, "\id_{\cat C}"'] \arrow[rr, "\alpha", ""name=iota] &        & \cat E \arrow[ld, "\pi"] \\
                                               & \cat C &      
    \arrow[to=2-2, from=iota, pos=.5, phantom, "\circ"]                    
    \end{tikzcd}
    $$
\end{cor}

\begin{proof}
    Under the correspondence given in Lemma~\ref{lem:discreteGrothendieck}, solutions to diagrams are one-to-one with functors $\cat C\xrightarrow{\alpha}\cat E$ as above. Explicitly, if $(x_C)_{C\in\cat C}$ is a solution of the diagram corresponding to $\pi$, then $\alpha:C\mapsto x_C$ defines the desired functor.
\end{proof}

\begin{exam}
    Such $\alpha$ can be thought of as ways for $\cat E$ to 'contain a copy' of $\cat C$. For example, let $\cat C$ be the class of finite linearly ordered sets, let $\cat E$ be the class of finite linearly ordered graphs, and let $\pi$ be the functor that forgets all edges. Then $\cat E$ contains two copies of $\cat C$, namely complete graphs and edgeless graphs.
\end{exam}

\begin{prop}
\label{prop:concrete expansions}
    If $\cat C$ is a class of relational structures, then every expansion of $\cat C$ in the sense of Definition~\ref{def:expansions} is isomorphic to an expansion in the sense of Definition~\ref{def:expansionExplicit}.
\end{prop}

\begin{proof}
    Given an abstract expansion $\cat E\xrightarrow{\pi}\cat C$, let $\tau$ be the set of objects in $\cat E$, where the arity of a symbol $E\in\tau$ is the underlying set of $\pi(E)$. For each symbol $E$ construct a $\sigma\cup\tau$-structure $\str E$ by defining the relations as follows. $\str E|_\sigma := \pi(E)$ and for any symbol $E'\in\tau$ define its interpretation $E'^{\str E}$ as the set of all functions from the domain of $\pi(E')$ to the domain of $\pi(E)$, which arise from the arrows of $\cat E$ through $\pi$.
    $$
    E'^{\str E} := \{\pi(E')\xrightarrow{\pi(g)}\pi(E)\mid E'\xrightarrow{g} E\}
    $$
    There is a bijection between objects in $\cat E$ and structures obtained this way and if $E'\xrightarrow{g}E$ is an arrow in $\cat E$, then the map $\pi(g)$ is an embedding $\str E'\to\str E$. Conversely, if a map $f$ is an embedding $\str E'\to\str E$, then the relation $E'^{\str E}$ must contain $f$, hence there is an  arrow $E'\xrightarrow{g} E$ with $f=\pi(g)$, which must be unique due to the second axiom of expansions.
\end{proof}

Let $\kappa$ be a regular cardinal, and let $L_{\kappa,0}$ be the logic which expands first-order logic by infinite conjunctions and disjunctions of size less than $\kappa$.

\begin{prop}
\label{prop:definability}
    Let $\cat E\xrightarrow{\pi}\cat C$ be an abstract expansion between two classes of relational structures, where both $\cat C$ and $\cat E$ are closed under taking induced substructures, and $\pi$ preserves the underlying sets and maps. Let the signatures of $\cat C$ and $\cat E$ be $\sigma$ and $\tau$ respectively. If $\kappa$ is a regular cardinal larger than the number of symbols in $\sigma$ and the number of non-isomorphic structures in $\cat E$, then for every symbol $R\in\sigma$ there is a quantifier free $L_{\kappa,0}$-formula $\phi_R$ whose set of variables is the arity of $R$, which defines $\pi$, i.e.
    $$
    \pi(\str E) = (E; \phi_R^{\str E}\mid R\in\sigma)\in\cat C
    $$
    for all $\str E\in\cat E$.
    In particular, one can $L_{\kappa,0}$-define $\cat C$ from $\cat E$ without quantifiers.
\end{prop}

\begin{proof}
    For each structure $\str E\in\cat E$, let $\psi_{\cat E}$ be the formula on variable set $E$, which determines $\str E$, i.e.
    $$
    \psi_{\str E} = 
    \bigwedge_{\substack{S\in\tau\\ \bar{e}\in S^{\str E}}} S(\bar{e}) \wedge
    \bigwedge_{\substack{S\in\tau\\ \bar{e}\notin S^{\str E}}} \neg S(\bar{e})
    $$
    Given a $\tau$-structure $\str E'$ and an injective tuple $E\xrightarrow{f}E'$, the statement $\psi_{\str E}(f)$ will be true in $\str E'$ if and only if $f$ is an embedding from $\str E$ to $\str E'$. To deal with non-injective tuples, for every surjective map $N\xrightarrow{g}M$ we define the following formula with variable set $N$.
    $$
    \psi_g =
    \bigwedge_{\substack{x,y\in N\\ g(x)= g(y)}} (x = y) \wedge
    \bigwedge_{\substack{x,y\in N\\ g(x)\neq g(y)}} (x \neq y)
    $$
    For any symbol $R\in\sigma$, let $\phi_R$ be the following formula with variable set $\ar R$:
    $$
    \phi_R :=
    \bigvee_{\str E\in\cat E}
    \bigvee_{\substack{\bar{e}\in R^{\pi(\str E)} \\ \text{surj}}}
    \psi_{\bar{e}}\wedge \psi_{\str E}(\iota)
    $$
    Here, $\bar{e}$ is any surjective tuple the relation $R^{\pi(\str E)}$, while $\iota$ is any map $E\xrightarrow{\iota}\ar R$ which is left inverse to $\bar{e}$, so $\bar{e}\circ\iota=\id_{E}$. The size of the above formulae is bounded by the size of the signature $\tau$ and the number of non-isomorphic structures in $\cat E$ whose domain size is not larger than the arity of $S$.
    
    To verify that the formulae $\phi_R$ define $\pi$, let $\str E$ be a structure in $\cat E$ and let $\ar S\xrightarrow{\bar{c}} E$ be a tuple. We show that $\bar{c}\in R^{\pi(\str E)}$ if and only if $\phi_R(\bar{c})$ is true in $\str E$.
    
    Assume that $\phi_R(\bar{c})$ is true in $\str E$, then there is a structure $\str E'$ and a surjective tuple $\bar{e}\in R^{\pi(\str E')}$ such that $\psi_{\bar{e}}(\bar{c})$ and $\psi_{\str E}({\bar{c}\circ\iota})$ are true in $\str E'$. The former implies that $e_i=e_j$ if and only if $c_i=c_j$ and in particular that $\bar{c}\circ\iota$ is injective. Hence, it is an embedding $\str E'\xrightarrow{}\str E$ and also an embedding $\pi(\str E')\xrightarrow{}\pi(\str E)$. This implies that $\bar{c}=\bar{c}\circ\iota\circ\bar{e}$ is a tuple in $R^{\pi(\str E)}$.

    Conversely, if $\bar{c}$ is an element of $R^{\pi(\str E)}$, let $\str E'$ be the substructure of $\str E$ induced on the image of $\bar{c}$, then $\bar{c}$ is a surjective tuple in $R^{\pi(\str E')}$. Moreover, both $\psi_{\bar{c}}(\bar{c})$ and $\psi_{\str E'}(\bar{c}\circ\iota)$ hold in $\str E$, since $\bar{c}\circ\iota$ is the inclusion $\str E'\leq \str E$.
\end{proof}

\subsection{Homomorphisms, cores and the expansion property}
\label{subsec:expansions}

If a class of relational structures has a precompact Ramsey expansion, then it has a unique minimal precompact Ramsey expansion. A similar statement first appeared in the context of topological dynamics in \cite[Theorem~9.2, Theorem~10.7]{KPT2005} 
and was later refined in \cite[Theorem~6]{NVT2013precompactExpansions}. We provide a further refinement and a generalization to arbitrary essentially small locally finite categories using abstract expansions and new auxiliary notions, namely homomorphisms and cores of expansions.

\begin{defin}
\label{def:homs}
    A \emph{homomorphism} between two expansions $\cat E\xrightarrow{\pi}\cat C$ and $\cat F\xrightarrow{\rho}\cat C$ is a functor $\cat E\xrightarrow{\alpha}\cat F$ such that $\pi=\rho\circ\alpha$.
    $$
    \begin{tikzcd}
    \cat E \arrow[rd, "\pi"'] \arrow[rr, "\alpha", ""name=alpha] &        & \cat F \arrow[ld, "\rho"] \\
    & \cat C &
    \arrow[to=2-2, from=alpha, pos=.5, phantom, "\circ"]
    \end{tikzcd}
    $$
    We say that $\alpha$ is surjective (bijective, \dots) if it is surjective (bijective, \dots) on objects.
    Observe that homomorphisms of expansions are invertible if and only if they are bijective.
\end{defin}

\begin{lem}
\label{lem:ramseyHoms}
Let $\cat E\xrightarrow{\pi}\cat C$ and $\cat F\xrightarrow{\rho}\cat C$ be two expansions. If $\cat E$ is confluent Ramsey and $\rho$ is precompact, then there is a homomorphism of expansions $\pi\to\rho$.
\end{lem}

\begin{proof}
    We define a diagram $\cat E\op\to\Set, E\mapsto \rho^{-1}(\pi(E))$, where the sets are finite and nonempty by precompactness and surjectivity of $\rho$, respectively. Hence the diagram has a solution by Theorem~\ref{thm:ramsey=könig}, which is the sought-after homomorphism.
\end{proof}

\begin{lem}
\label{lem:surj}
    Surjective homomorphisms of expansions are themselves expansions.
\end{lem}

\begin{proof}
    One needs to verify that they are discrete fibrations. Using the notation form Definition~\ref{def:homs}, let $E\in\cat E$ be an object and $F\xrightarrow{f}\alpha(E)$ an arrow in $\cat F$. Since $\pi$ is an expansion, there is a unique arrow $E'\xrightarrow{g}E$ with target $E$ such that $\pi(g)=\rho(f)$. This in turn implies $\alpha (g) =f$, as there is a unique arrow in $\cat F$ whose target is $\alpha( E)$ and whose $\pi$-image is $\rho(f)$.
\end{proof}

We show that if an expansion has the so-called \emph{expansion property}, then every homomorphism to it is surjective. Combined with the two above lemmas, this will yield that every Ramsey expansion of $\cat C$ is an expansion of every expansion of $\cat C$ with the expansion property.

\begin{defin}[Definition 2 of \cite{NVT2013precompactExpansions}]
\label{def:ep}
Let $\cat E\xrightarrow{\pi}\cat C$ be an expansion and let $C\in\cat C$. We say that $C'\in\cat C$ has the \emph{expansion property} (EP) with respect to $C$, if for all $E\in\pi^{-1}(C)$ and $E'\in\pi^{-1}(C')$ there is an arrow $E\xrightarrow{}E'$ in $\cat E$. We say that $\pi$ has EP if for every $C$ there exists such $C'$.
\end{defin}

\begin{prop}
\label{prop:ep}
Let $\cat E\xrightarrow{\pi}\cat C$ be a precompact expansion and let $\cat C$ be confluent. Then $\pi$ has EP if and only if all homomorphisms to it are surjective.
\end{prop}        

\begin{proof}
Let $\cat F\xrightarrow{\rho}\cat C$ be another expansion and let $\cat F\xrightarrow{\alpha}\cat E$ be a homomorphism from $\rho$ to $\pi$. 

First, we assume that $\pi$ has EP and show that any object $E\in\cat E$ is in the image of $\alpha$. Let $C'$ be an object with the expansion property with respect to $\pi(E)$ and let $F'\in\cat F$ be an object with $\rho(F')=C'$. Since $\pi(\alpha(F'))=C'$ there must be an arrow $E\xrightarrow{f}\alpha(F')$ in $\cat E$. Also, $\rho$ is an expansion; hence, there is an arrow $F\xrightarrow{g}F'$ with $\pi(f)=\rho(g)=\pi(\alpha(g))$. Since $f$ and $\alpha(g)$ share the same target, namely $\alpha(F')$, and have the same image under $\pi$, they must be equal since $\pi$ is an expansion. In particular, they share the same source $\alpha(F)=E$.

For the converse, we fix $C\in\cat C$ and $E\in\cat \pi^{-1}(C)$. 
Let $\cat E'$ be the subcategory of $\cat E$ consisting of all objects which do not admit an arrow from $E$. 
If $\cat E'\xrightarrow{\pi|_{\cat E'}}\cat C$ were an expansion, then the inclusion $\cat E'\leq \cat E$ would be a homomorphism to $\pi$, which is not surjective as $E\notin\cat E'$.
However, $\cat E'\xrightarrow{\pi|_{\cat E'}}\cat C$, is a discrete fibration, so it cannot be surjective. In particular, there is $C_E\in\cat C$ such that every $E'\in\pi^{-1}(C_E)$ admits an arrow $E\to E'$. 
This also gives an arrow $C\to C_E$ and since $\cat C$ is confluent, we find an object $C'$ such that there are arrows $C_E\to C'$ for every $E\in\cat E_C$. 

This $C'$ has the expansion property with respect to $C$.
\end{proof}

\begin{cor}
\label{cor:order}
    Let $\cat C$ be a confluent Ramsey class of relational structures which is closed under taking induced substructures and isomorphisms. Then one can $L_{\kappa,0}$-define the class of linear orders from $\cat C$, where $\kappa$ is any regular cardinal larger than the language of $\cat C$ and the number of non-isomorphic structures in $\cat C$ whose domain consists of two elements. 
\end{cor}

\begin{proof}
    If $\cat C$ is closed under taking induced substructures and isomorphism, then it is an expansion of the class $\cat I$ of finite sets and injections. The class $\cat L$ of linearly ordered sets is a precompact expansion of $\cat I$, hence Lemma~\ref{lem:ramseyHoms} provides a homomorphism.
    $$
    \begin{tikzcd}
    \cat C \arrow[rd, "\rho"'] \arrow[rr, two heads,"\alpha", ""name=alpha] &        & \cat L \arrow[ld, "\pi"] \\
    & \cat I & 
    \arrow[to=2-2, from=alpha, pos=.5, phantom, "\circ"]
    \end{tikzcd}
    $$
    The expansion $\pi$ has EP, therefore $\alpha$ is surjective by Proposition~\ref{prop:ep} and an expansion by Lemma~\ref{lem:surj}. The claim follows from Proposition~\ref{prop:definability}.
\end{proof}

\begin{defin}
\label{def:core}
    We call an expansion $\cat E\xrightarrow{\pi}\cat C$ a \emph{core}, if all its endomorphisms are bijective. If $\cat E\xrightarrow{\pi}\cat C$ is an arbitrary expansion, if $\cat F\xrightarrow{\rho}\cat C$ is a core and if there are homomorphisms $\alpha$ and $\beta$ as in the following diagram, then we say that $\rho$ is \emph{the} core of $\pi$.
    $$
    \begin{tikzcd}
\cat F \arrow[r, "\alpha"] \arrow[rd, "\rho"']  
& 
\cat E \arrow[r, "\beta"] \arrow[d, "\pi"] & 
\cat F \arrow[ld, "\rho"] \\
 & \cat C   &                            
\end{tikzcd}
    $$
    The core of $\rho$ is unique up to isomorphism and (if it exists) it is given by a minimal image of endomorphisms of $\rho$. 
\end{defin}

\begin{exam}
    Let $\cat I$ be the category of sets and injective maps, and let $\cat L$ be the category of linearly ordered sets and embeddings. The expansion $\cat L\xrightarrow{\pi}\cat I$ is a core, as its only (non-identity) endomorphism is the reversal of the order $(N,<)\mapsto(N,>)$.
\end{exam}

\begin{lem}
\label{lem:coreSurj}
    Let $\cat E\xrightarrow{\pi}\cat C$ be a precompact expansion where $\cat C$ is confluent and $\cat E$ is confluent Ramsey. Then $\pi$ is a core if and only if every homomorphism of expansions to it is surjective. 
\end{lem}

\begin{proof}
    Assume that $\cat E\xrightarrow{\pi}\cat C$ is a core and let $\alpha$ be a homomorphism of expansions to $\pi$. Then $\im(\alpha)\xrightarrow{\pi}\cat C$ is a precompact expansion, hence there is a homomorphism from $\pi$ to it, i.e. there is an endomorphism of $\pi$ whose image is contained in the image of $\alpha$. But every endomorphism is bijective since $\pi$ is a core, therefore $\alpha$ must be surjective.

    Conversely, if every homomorphism to $\pi$ is surjective, then in particular every endomorphism of $\pi$ is surjective and surjective endomorphisms of precompact expansions are bijective.
\end{proof}

\begin{thm}
\label{thm:cores}
    Every precompact, confluent Ramsey expansion has a core.
\end{thm}

\begin{proof}
Let $\cat E\xrightarrow{\pi}C$ be such an expansion and consider the poset of all images of its endomorphisms, ordered by inclusion.
$$
\cat P:=\{\im(\alpha)\leq\cat E\mid \alpha\text{ endomorphism of } \pi \}
$$
We use Zorn's lemma to find a minimal element of this poset, which will be the core of $\pi$.
Consider a descending chain 
$$
\cat E\geq \cat E^1\geq \cat E^2\geq \cdots 
$$
in $\cat P$. 
We claim that the intersection $\cat E^* :=\bigcap_{n\in\mathbb{N}}\cat E^n\leq\cat E$ is still an expansion of $\cat C$. Indeed, if $C\in\cat C$, then $\pi^{-1}(C)\cap \cat E^n$ is a sequence of finite nonempty sets, whose intersection is nonempty by K\H{o}nig's lemma.
By Lemma \ref{lem:ramseyHoms}, there is a homomorphism of expansions
    $$
    \begin{tikzcd}
    \cat E \arrow[rd, "\pi"'] \arrow[rr, "\alpha", ""name=alpha] &        &  \cat E^*\arrow[ld, "\pi"] \\
    & \cat C & 
    \arrow[to=2-2, from=alpha, pos=.5, phantom, "\circ"]
    \end{tikzcd}
    $$
In particular, $\alpha$ is an endomorphism of $\pi$ whose image is contained in all $\cat E^n$, so by Zorn's lemma there is a minimal element of $\cat P$ which is the core of $\pi$.
\end{proof}

\begin{cor}
\label{cor:uniqueMinimal}
    If $\cat C$ has a precompact Ramsey expansion, then it has a minimal precompact Ramsey expansion $\cat E\xrightarrow{\pi}\cat C$, which is unique up to isomorphism. Minimality means that for every other Ramsey expansion $\cat F\xrightarrow{\rho}\cat C$ there is an expansion $\cat F\xrightarrow{\alpha}\cat E$ with $\pi\circ\alpha=\rho$.
    $$
    \begin{tikzcd}
    \cat F \arrow[rd, "\rho"'] \arrow[rr, "\alpha", ""name=alpha] &        & \cat E \arrow[ld, "\pi"] \\
    & \cat C &
    \arrow[to=2-2, from=alpha, pos=.5, phantom, "\circ"]
    \end{tikzcd}
    $$
\end{cor}

\begin{proof}
    Let $\cat E\xrightarrow{\pi}\cat C$ be the core of a precompact Ramsey expansion which exists by Theorem~\ref{thm:cores}. Ramseyness of $\cat E$ will follow from a general principle which we prove in Proposition~\ref{prop:preadjunction} in the next section. Any other confluent Ramsey expansion of $\cat C$ will have a homomorphism to $\pi$ by Lemma~\ref{lem:ramseyHoms} that is an expansion by Lemmas~\ref{lem:surj} and \ref{lem:coreSurj}.
\end{proof}

\begin{cor}
    Let $\cat C$ be a class of relational structures with a precompact Ramsey expansion $\cat E\xrightarrow{\pi} \cat C$ that has the expansion property.
    Then every other Ramsey expansion $\cat F\xrightarrow{\rho}\cat C$ can $L_{\kappa,0}$-define $\cat E$ without quantifiers, where $\kappa$ is any regular cardinal larger than the number of symbols in the language of $\cat F$ and the number of non-isomorphic structures in $\cat F$. In particular, $\cat E$ is unique up to quantifier-free $L_{\kappa,0}$-interdefinability.
\end{cor}

%% file: content/transfers.tex
In this section, we discuss how Theorem~\ref{thm:ramsey=könig} might be used to transfer the Ramsey property from one category to another.
We assume that all the categories mentioned are essentially small and locally finite, so that Theorem~\ref{thm:ramsey=könig} applies.

\subsection{The trivial transfer}

\begin{prop}
\label{prop:preadjunction}
    Let $\cat B$ be a confluent Ramsey category and let $\cat C$ be another category. If there are functors $F$ and $G$ and a natural transformation $\Delta$ as in the following diagram, then $\cat C$ is also confluent and Ramsey.
    $$
 \begin{tikzcd}
\cat C  \arrow[rr, "\id_{\cat C}", ""name=id] 
        \arrow[rd, "F"'] &   & 
\cat C \\
& \cat B \arrow[ru, "G"']       & \arrow[from=id, to=2-2, "\Delta", Rightarrow, shorten >=1ex, shorten <=1ex]      
\end{tikzcd}
$$
\end{prop}

\begin{proof}
    Take a diagram $\cat D:\cat C\op\to\Set$ such that $\cat D_C$ is finite and nonempty for any $C\in\cat C$.  Then $\cat D_{G(B)}$ is finite and nonempty for any $B\in\cat B$, implying that the diagram $\cat D\circ G\op:B\op\to\Set $ has a solution. Moreover, $F$ and $\Delta$ respectively induce maps
    $$
    \emptyset\neq\lim(\cat D\circ G\op)
    \xrightarrow{\tilde{F}}\lim(\cat D\circ G\op\circ F\op) 
    \xrightarrow{\Tilde{\Delta}} \lim\cat D
    $$
    implying that also $\cat D$ has a solution. Explicitly, if $(x_B)_{B\in\cat B}$ is a solution of $\cat D\circ G\op$, then $(x_{F (C)})_{C\in\cat C}$ is a solution of $\cat D\circ G\op\circ F\op$ and $\big(\cat D_{\Delta_C}(x_{F(C)})\big)_{C\in\cat C}$ is a solution of $\cat D$.
\end{proof}
\begin{rem}
Proposition~\ref{prop:preadjunction} connects three known Ramsey transfer principles, namely range-rigid functions \cite[Lemma~10]{mottet2021cores}, semi-retractions  \cite[Theorem~3.5]{scow2021semiretractions} and pre-adjunctions \cite[Theorem 3.2]{mavsulovic2018pre}. 
Indeed, one can show that the former two are both equivalent to the special case of Proposition~\ref{prop:preadjunction} when $\cat B$ and $\cat C$ are classes of relational structures, $F$ and $G$ preserve the underlying sets and $\alpha$ is the identity.

Theorem~3.2 in \cite{mavsulovic2018pre} is more general than Proposition~\ref{prop:preadjunction}, in particular $F$, $G$ and $\alpha$ as above induce a pre-adjunction in the following way.
$$
\hom(F(C), B) \xrightarrow{G}
\hom(G(F(C)),G(B))\xrightarrow{\Delta_C^*} 
\hom(C,G(B))
$$
However, all examples of pre-adjunctions in \cite{mavsulovic2018pre} are induced in this way.
\end{rem}

\begin{cor}
    Let $\cat E\xrightarrow{\pi}\cat C$ be an expansion where $\cat E$ is confluent Ramsey, and let $\cat F\xrightarrow{\rho}\cat C$ be its core (see Definition~\ref{def:core}). Then $\cat F$ is confluent and Ramsey.
\end{cor}

\begin{proof}
    By definition, there are homomorphisms $\alpha$ and $\beta$ of expansions as in the following diagram.
    $$
    \begin{tikzcd}
    \cat F \arrow[r, "\alpha"] \arrow[rd, "\rho"']
    & 
    \cat E \arrow[r, "\beta"] \arrow[d, "\pi"] & 
    \cat F \arrow[ld, "\rho"] \\
    & \cat C   &                            
\end{tikzcd}
    $$
Since $\rho$ is a core, the functor $\beta\circ \alpha$ is invertible and the claim follows from Proposition~\ref{prop:preadjunction} and the following diagram.
    $$
    \begin{tikzcd}
    \cat F \arrow[rd, "\alpha"'] \arrow[rr, "\id_{\cat F}", ""name=alpha] &        &  \cat F \\
    & \cat E\arrow[ru, "(\beta\circ \alpha)^{-1}\circ \beta"'] & 
    \arrow[to=2-2, from=alpha, pos=.5, phantom, "\circ"]
    \end{tikzcd}
    $$
\end{proof}

\subsection{Opfibrations}
We prove a new Ramsey transfer theorem, namely that a Grothendieck opfibration is Ramsey whenever its base and its fibers are confluent and Ramsey. This provides a common generalization of three known transfers; products \cite{sokic2012ramsey}, discrete opfibrations \cite{mavsulovic2023fibrations} (hence slice categories and adding constants \cite{bodirsky2013decidability}) and blowups. We use the latter as an illustration of the Grothendieck construction.

\begin{exam}
\label{ex:partitions}
Consider the following category $\cat C$ of set partitions. Objects in $\cat C$ are tuples $(N, (M_n)_{n\in N})$, where $N$ is a set of blocks and $M_n$ is the $n$-th block. Arrows 
$$
(N, (M_n)_{n\in N})\to(N', (M_{n'}')_{n'\in N'})
$$ 
in $\cat C$ are tuples $(f,(g_n)_{n\in N})$ where $N\xrightarrow{f}N'$ is an injective map and every $g_n$ is an injective map $M_n\to M'_{f(n)}$.

$\cat C$ has a partition into subclasses, for every set $N$ let $\cat C_N$ consist of partitions with exactly $N$ blocks. For each injective map $N\xrightarrow{f}N'$, define a functor $\cat C_N\to\cat C_{N'}$ which adds empty blocks for all elements $n'\in N'$ which are not in the image of $f$ and leaves all other blocks unchanged. The class $\cat C$ can be recovered from the fibers $\cat C_N$ and the functors obtained from the $f$'s using the following construction. 
\end{exam}

\begin{defin}[Grothendieck construction]
    Consider a functor $\cat S: \cat R\to \Cat $ form a category $\cat R$ to the category of categories. We define a new category $\Elts\cat S$ as follows. The objects are pairs $(R,S)$, where $R$ is an object in $\cat R$ and $S$ is an object in $\cat S_R$. The arrows between pairs $(R,S)$ and $(R',S')$ are pairs $(f,\phi)$, where $R\xrightarrow{f}R'$ is an arrow in $\cat R$ and $\cat S_f(S)\xrightarrow{\phi} S'$ is an arrow in $\cat S_{R'}$.
    \begin{align*}
        \mathrm{Ob}(\textstyle\Elts\cat C) &= \{ ( R,S) \mid R\in\cat R, S\in\cat S_R\} \\
        \hom((R,S),(R',S')) &= \{ (f, \phi)\mid R\xrightarrow{f}R', \cat S_f(S)\xrightarrow{\phi} S' \}
    \end{align*}
    Categories constructed this way are called \emph{opfibrations} over $\cat R$, the categories $\cat S_R$ are called its \emph{fibers}, see, for example, \cite{LoregianRiehl2020fibrations}.
\end{defin}

\begin{exam}
    Let $\cat I$ be the category of finite sets and injective maps, and let $\cat S:\cat L\to\Cat$ send $N$ to $\cat I^N$. For an injective map $N\xrightarrow{f}N'$, let the functor $\cat I^N\xrightarrow{\cat S_f}\cat I^{N'}$ add empty blocks as in Example~\ref{ex:partitions}. Then $\Elts\cat S$ is the category of partitions described above.
\end{exam}

\begin{thm}
\label{thm:opfibrations}
    Let $\cat R$ be a confluent Ramsey category and let $\cat S:\cat R\to \Cat$ be a functor such that also $\cat S_R$ is a confluent Ramsey category for all objects $R\in \cat R$. If all above categories are essentially small and locally finite, then $\Elts\cat S$ is confluent and Ramsey.
\end{thm}

Let us first sketch the proof. We need to show that diagrams in the shape of $\Elts\cat S$ have a solution, a tuple compatible with all arrows. The partition of $\Elts\cat S$ into its fibers $\cat S_R$ distinguishes two classes of arrows: the ones inside the fibers 
$$(R,S)\xrightarrow{(\id_R,\phi)}(R,S')$$ 
and the ones 'orthogonal' to the fibers 
$$(R,S)\xrightarrow{(f,\id)} (R', \cat S_f(S)).$$ 
The first kind is no problem, every fiber is confluent Ramsey, hence there are partial solutions on every fiber. To deal with the second kind of arrows, note that any $R\xrightarrow{f}R'$ induces a map from partial solutions on $\cat S_{R'}$ to partial solutions on $\cat S_R$ similar to Proposition~\ref{prop:preadjunction}.
$$
    (x_{S'})_{S'\in\cat S_{R'}}\mapsto 
    \big(\cat D_{f,\id} (x_{\cat S_f(S)}) \big)_{S\in\cat S_R}
$$
This defines a diagram $\cat R\op\to \Set,R\mapsto \lim \cat D|_{\cat S_R}$ whose solutions would correspond to solutions of $\cat D$. However, there might be infinitely many partial solutions to $\cat S_R$, hence the entire argument must be restricted to finite subcategories.

\begin{proof}
    Let $\cat D:(\Elts\cat S)\op\to \Set$ be a diagram, and let $M$ be a finite set of arrows in $\Elts\cat S$. We show that $\cat D$ has an $M$-solution which will prove the claim by Lemma~\ref{lem:compactness1} and Theorem~\ref{thm:ramsey=könig}.

    Since $M$ is finite, there is a finite list $(R_i)$ of objects in $\cat R$ and finite lists $(S_i^j)$ of objects in $\cat S_{R_i}$ such that every arrow in $M$ is between pairs of the form $(R_i,S_i^j)$. 

    Let $\hat{\cat S}:\cat R\to\Cat $ be the functor that sends $R\in \cat R$ to the full subcategory $\hat{\cat S}_R\leq\cat S_R$, whose objects are of the form $\cat S_f(S_i^j)\in \cat S_R$, where $R_i\xrightarrow{f}R$ is any morphism from some $R_i$ to $R$. Note that for each $R$ there are only finitely many such objects, as there are finitely many $R_i$ and hence only finitely many such arrows $f$. 
    Clearly $\Elts\hat{\cat S}$ is a subcategory of $\Elts\cat S$ which contains all morphisms in $M$, so if there is a partial solution to $\cat D$ on $\Elts\hat{\cat S}$, then there is also an $M$-solution.

    To find a partial solution to $\cat D$ on $\Elts\hat{\cat S}$, define a diagram $\hat{\cat D}: \cat R\op\to\Set$ that sends $R$ to $\lim\cat D|_{\hat{\cat S}_R}$, the set of partial solutions to $\cat D$ on $\hat{\cat S}_R$.
    This set is finite, since $\hat{\cat S}_R$ contains only finitely many objects and nonempty, because there are partial solutions to $\cat D$ on $\cat S_R$. In particular, there are partial solutions on $\hat{\cat S}_R$.
    For arrows $R\xrightarrow{f}R'$, we define $\hat{\cat D}_f:\lim\cat D|_{\hat{\cat S}_{R'}}\to \lim\cat D|_{\hat{\cat S}_R}$ as the following map. 
    $$
    (x_{S'})_{S'\in\hat{\cat S}_{R'}}\mapsto 
    \big(\cat D_{f,\id} (x_{\cat S_f(S)}) \big)_{S\in\hat{\cat S} (R)} 
    $$
    Since $R$ is also confluent and Ramsey, $\hat{\cat D}$ will have a solution
    $$
    \big((x_{R,S})_{S\in\hat{\cat S}_R}\big)_{R\in \cat R}
    $$
    where for each $R\in\cat R$, the tuple $(x_{R,S})_{S\in\hat{\cat S}_R}$ is a partial solution of $\cat D$ on $\hat{\cat S}_R$. To verify that $(x_{R,S})_{(R,S)\in\Elts\hat{\cat S}}$ is a partial solution of $\cat D$, take an arrow $(R,S)\xrightarrow{(f,\phi)}(R',S')$ in $\Elts\hat{\cat S}$ and compute the following.
    \begin{align*}
    \cat D_{f,\phi}(x_{R',S'})&=
    \cat D_{f,\id} \big(\cat D_{\id,\phi} (x_{R',S'})\big) \\&=  
    \cat D_{f,\id} (x_{R',\cat S_f(S)}) \\&=
    x_{R,S}
    \end{align*}
    The second equality holds because $(X_{R',S'})_{S'\in \cat S_{R'}}$ is a partial solution on the fiber $\cat S_{R'}$ and the third holds by definition of $\hat{\cat D}$.
\end{proof}

\begin{cor}[Theorem 2 in \cite{sokic2012ramsey}]
\label{cor:prod}
    If $\cat C$ and $\cat D$ are two essentially small locally finite confluent Ramsey categories, then their product $\cat C\times\cat D$ is also confluent and Ramsey.
\end{cor}

\begin{proof}
    If $\cat S:\cat C\to\Cat$ is the constant functor $C\mapsto \cat D$, then $\Elts \cat S=\cat C\times \cat D$.
\end{proof}

\begin{cor}[Theorem 5.1.c in \cite{mavsulovic2023fibrations}]
\label{cor:slice }
    If $\cat C$ is an essentially small locally finite confluent Ramsey category and $\cat S:A\to\Set$ a functor, then $\Elts \cat S$ is also confluent and Ramsey. In particular, if $A\in\cat C$ is an object, then the slice category $A\backslash\cat C$ is confluent Ramsey.
\end{cor}
\begin{proof} 
Since discrete categories are confluent Ramsey, the first part follows directly from Theorem~\ref{thm:opfibrations}. For the second part, take $\cat S:C\mapsto\hom(A,C)$, then $\Elts\cat S=A\backslash\cat C$ is the slice category.
\end{proof}

\begin{defin}
    Given two classes of relational structures $\cat C$ and $\cat D$, their \emph{blowup} $\cat C\rtimes\cat D$ is the following category. Objects of $\cat C\rtimes\cat D$ are tuples $(\str C,(\str D_c)_{c\in C})$ where $\str C\in\cat C$ and $\str D_c\in\cat D$. Arrows between two objects $(\str C,(\str D_c)_{c\in C})$ and $(\str C',({\str D'}_{c'})_{c'\in C'})$ are tuples $(f, (g_c)_{c\in C})$ where $f$ is an embedding $\str C\to\str C'$ and each $g_c$ is an embedding $\str D_c\to \str D'_{f(c)}$.
\end{defin}

The following corollary is known, though we were unable to find a direct reference. It follows for example from Theorem~6.1 and Lemma~6.7 in \cite{KPT2005}.

\begin{cor}
    Let $\cat C$ and $\cat D$ be confluent Ramsey classes of relational structures. Then the blowup $\cat C\rtimes \cat D$ is confluent and Ramsey.
\end{cor}

\begin{proof}
    The blowup $\cat C\rtimes\cat D$ can be constructed as $\Elts\cat S$, where $\cat S$ is the following functor.
    $$
    \cat S:\cat C\to\Cat , \quad\str C\mapsto \cat D^{C}
    $$
    Given an embedding $\str C\xrightarrow{f}\str C'$ in $\cat C$ let $\cat S_f$ be the following functor.
    $$
    \cat S_f : 
    \begin{cases}
        \cat D^C\to\cat D^{C'} \\
        (\str D_c)_{c\in C} \mapsto (\tilde{\str D}_{c'})_{c'\in C'}   
    \end{cases} \text{where }
    \tilde{\str D}_{c'} = 
    \begin{cases}
    \str D_c &\text{if } c'=f(c)\\
    \emptyset &\text{if }c'\notin \mathrm{im}(f)
    \end{cases}
    $$
    Here $\emptyset$ denotes the unique structure on the empty set.
    Each $\cat D^{C}$ is confluent and Ramsey by Corollary~\ref{cor:prod}, hence so is $\Elts\cat S=\cat C\rtimes\cat D$.
\end{proof}

\subsection{Free superpositions}
We include one more transfer theorem from Bodirsky \cite{bodirsky2012transfers}, namely that the free superposition of Ramsey classes is again Ramsey, assuming the strong amalgamation property. We consider the presented proof to be conceptually interesting, it uses the trivial transfer (Proposition~\ref{prop:preadjunction}), but one of the two functors is obtained from Theorem~\ref{thm:ramsey=könig} rather than constructed explicitly.
\begin{defin}
    We say that a class of structures $\cat C$ has the \emph{strong amalgamation property} if the following holds. Given three structures $\str A$, $\str B$, $\str C$ in $\cat C$, two embeddings $\str B\xleftarrow{f}\str A\xrightarrow{g}\str C$ and two injective maps $ B\xrightarrow{h}D\xleftarrow{k}C$ such that $h\circ f= k\circ g$, and $\mathrm{im}(h)\cap\mathrm{im}(k) = \mathrm{im}(h\circ f)$, there is a structure $\str D\in\cat D$ with domain $D$, such that $h$ and $k$ are embeddings $\str B\to \str D$ and $\str C\to\str D$ respectively.
\end{defin}

\begin{defin}
    Let $\cat C$ and $\cat D$ be two classes of finite relational structures in disjoint signatures $\sigma$ and $\tau$, respectively. Let $\cat C\wedge\cat D$ be the class of all $\sigma\cup\tau$-structures, whose $\sigma$-reduct is in $\cat C$ and whose $\tau$-reduct is in $\cat D$. The class $\cat C\wedge\cat D$ is called the \emph{free superposition} of $\cat C$ and $\cat D$.
\end{defin}

\begin{prop}[Theorem~1.5 in \cite{bodirsky2012transfers}]
\label{prop:superposition}
    Let $\cat C$ and $\cat D$ be two Ramsey classes of finite relational structures, which are closed under taking induced substructures and contain only finitely many structures on any fixed domain. If both $\cat C$ and $\cat D$ have the strong amalgamation property, then the free superposition $\cat C\wedge \cat D$ is also Ramsey.
\end{prop}

\begin{lem}
\label{lem:superposition}
If a class $\cat C$ of relational structures has the strong amalgamation property, then for a given structure $\str C\in\cat C$ and a surjective map $D\xrightarrow{\pi} C$, there is a structure $\str D\in\cat C$ with domain $D$, such that any right inverse $C \xrightarrow{s} D$ of $\pi$ (i.e. any map $s$ with $\pi\circ s=\id_C$) is an embedding $\str C\to \str D$.
\end{lem}

\begin{proof}
We use induction on the size of $D$. As a first step, observe that the statement is true if $D$ has the same cardinality as $C$. 

If $D$ is larger than $C$, any surjective map $\pi:D\to C$ factors through a set $D'$ that contains exactly one element less than $D$.
$$
\begin{tikzcd}
D \arrow[r, "p"', two heads] \arrow[rr,"\pi", two heads, bend left] & D' \arrow[r, "q"', two heads] & C
\end{tikzcd}
$$
By the induction hypothesis, there is a structure $\str D'$ with domain $D'$ such that every right inverse of $q$ is an embedding $\str C\to \str D'$.
Now, let $s$ and $s'$ be the two right inverses of $p$ and let $\str D''$ be the substructure of $\str D'$ on the subset of $D'$ on which $s$ and $s'$ agree. By the strong amalgamation property of $\cat C$, there is a structure $\str D$ with domain $D$, such that both $s$ and $s'$ are embeddings $\str D'\to \str D$.
$$
\begin{tikzcd}
\str D' \arrow[r,"s"]       & \str D                   \\
\str D'' \arrow[u,hook] \arrow[r,hook] &\str D' \arrow[u, "s'"']
\end{tikzcd}
$$
We claim that $\str D$ is the desired structure. In fact, any right inverse of $\pi$ decomposes into a right inverse of $q$ followed by either $s$ or $s'$, all of which are embeddings.
\end{proof}

\begin{proof}[Proof of Proposition~\ref{prop:superposition}]
Note that there is the following inclusion functor. 
$$
\cat C\wedge \cat D\to \cat C\times \cat D, \quad\str E\mapsto(\str E|_\sigma, \str E|_\tau)$$
The aim is to find a functor $G:\cat C\times \cat D\to \cat C\wedge \cat D$ and a natural transformation~$\Delta$ as in the following diagram, then the claim will follow from Proposition~\ref{prop:preadjunction} and Corollary~\ref{cor:prod}.
$$
\begin{tikzcd}
\cat C\wedge \cat D \arrow[rd, hook] \arrow[rr, "\id_{\cat C\wedge \cat D}"] & {} \arrow[d, "\Delta"', Rightarrow] & \cat C\wedge \cat D\\
& \cat C\times \cat D \arrow[ru, "G"']     &    
\end{tikzcd}
$$
To that end, for two given structures $\str C\in \cat C$ and $\str D\in \cat D$ we consider structures $\str E\in \cat C\wedge\cat D$ on the domain $C\times D$ with the following properties.
\begin{enumerate}
    \item Every right inverse $C\xrightarrow{s}C\times D$  of the  first projection is an embedding $\str C\to\str E|_\sigma$. 
    \item Every right inverse $D\xrightarrow{t}C\times D$ of the second projection is an embedding $\str D\to\str E|_\tau$.
\end{enumerate}
Such $\str E$ exist by Lemma~\ref{lem:superposition}.
Now, define a diagram $\cat F:(\cat C\times \cat D)\op\to\Set$, which sends pairs $(\str C,\str D)$ to the set of all structures $\str E$ that fulfill 1) and 2). Given a pair of embeddings $\str C'\xrightarrow{f}\str C$ and $\str D'\xrightarrow{g} \str D$ we define the map $\cat F_{f,g}$ to send a structure $\str E$ with domain $C\times D$ to the unique structure $\str E'$ with domain $C'\times D'$ such that the pair $(f,g)$ is an embedding $\str E'\to \str E$. One may check that $\str E'$ still fulfills properties 1) and 2).

The sets $\cat F_{\str C,\str D}$ are finite because there are only finitely many structures in $\cat C\wedge\cat D$ with domain $C\times D$, and nonempty by Lemma~\ref{lem:superposition}. Hence, it has a solution by Theorem~\ref{thm:ramsey=könig} and this solution is a functor $G$ as above. To define the natural transformation $\Delta$, let $\str E\in\cat C\wedge\cat D$, then $\cat G_{\str E}$ is a structure on domain $E\times E$. The diagonal map $E\to E\times E$ is right inverse to both projections, hence it is an embedding $\str E\to \cat G_{\str E}$. Let $\Delta_{\str E}$ be this embedding.
\end{proof}

%% file: content/acknowledgements.tex
The author was funded by the European Unions ERC Synergy Grant 101071674, POCOCOP.
Views and opinions expressed are however those of the author(s) only and do not necessarily reflect those of the European Union or the European Research Council Executive Agency. Neither the European Union nor the granting authority can be held responsible for them.

Moreover, we would like to thank Mat\v{e}j Kone\v{c}n\'y for his remarks, his insights and for being an amazing source of inspiration.